\documentclass[12pt]{amsart}
\usepackage{graphicx}
\usepackage{graphics}
\usepackage{amsmath}
\usepackage{amscd}
\usepackage{latexsym}
\usepackage{amssymb}
\usepackage{amsthm}
\usepackage{dsfont}

\usepackage{fancyhdr}
\usepackage[shortlabels]{enumitem}

\usepackage{tikz}
\usetikzlibrary{arrows,cd,matrix,positioning,decorations.pathmorphing}
\usepackage{setspace}
\tikzset{snake arrow/.style=
{->,
decorate,
decoration={snake,amplitude=.4mm,segment length=2mm,post length=1mm}},
}

\usepackage{blindtext}
\usepackage{scrextend}
\addtokomafont{labelinglabel}{\sffamily}

\usepackage{hyperref}  % this is always the last package
\hypersetup{
  colorlinks   = true, %Colours links instead of ugly boxes
  urlcolor     = blue, %Colour for external hyperlinks
  linkcolor    = blue, %Colour of internal links
  citecolor    = blue  %Color of citations=
}

\begin{document}

\textwidth 6.2in
\textheight 7.6in
\evensidemargin.75in
\oddsidemargin.75in

\newtheorem{Thm}{Theorem}
\newtheorem{Lem}[Thm]{Lemma}
\newtheorem{Cor}[Thm]{Corollary}
\newtheorem{Prop}[Thm]{Proposition}
\newtheorem{Rm}{Remark}
\newtheorem{Qu}{Question}
\newtheorem{Def}{Definition}
\newtheorem{Exm}{Example}

\def\a{{\mathbb a}}
\def\C{{\mathbb C}}
\def\A{{\mathbb A}}
\def\B{{\mathbb B}}
\def\D{{\mathbb D}}
\def\E{{\mathbb E}}
\def\R{{\mathbb R}}
\def\P{{\mathbb P}}
\def\S{{\mathbb S}}
\def\Z{{\mathbb Z}}
\def\O{{\mathbb O}}
\def\H{{\mathbb H}}
\def\V{{\mathbb V}}
\def\Q{{\mathbb Q}}
\def\Cn{${\mathcal C}_n$}
\def\CM{\mathcal M}
\def\CG{\mathcal G}
\def\CH{\mathcal H}
\def\CT{\mathcal T}
\def\CF{\mathcal F}
\def\CA{\mathcal A}
\def\CB{\mathcal B}
\def\CD{\mathcal D}
\def\CP{\mathcal P}
\def\CS{\mathcal S}
\def\CZ{\mathcal Z}
\def\CE{\mathcal E}
\def\CL{\mathcal L}
\def\CV{\mathcal V}
\def\CW{\mathcal W}
\def\IC{\mathbb C}
\def\IF{\mathbb F}
\def\IK{\mathcal K}
\def\IL{\mathcal L}
\def\IP{\bf P}
\def\IR{\mathbb R}
\def\IZ{\mathbb Z}

\title{Knot concordances in $S^1\times S^2$ and exotic smooth $4$-manifolds }
\author{Selman Akbulut}
\author{Eylem ZEL\.{I}HA Yildiz}
\thanks{The first author is partially supported by NSF grant DMS 0905917}
\address{Department  of Mathematics, Michigan State University,  MI, 48824}
\email{akbulut@msu.edu}
\email{ezyildiz@msu.edu}
\date{\today}
\begin{abstract} 

It is known that there is a unique concordance class in the free homotopy class of $S^1\times pt \subset S^1 \times S^2$. The constructive proof of this fact is given by the second author. It turns out that all the concordances in this construction are invertible. The knots $K\subset S^{1}\times S^{2}$ with hyperbolic complements and trivial symmetry group are special interest here, because they can be used to generate absolutely exotic compact 4-manifolds by the recipe given by  Akbulut and  Ruberman. Here we built absolutely exotic manifold pairs by this construction, and show that this construction keeps the Stein property of the $4$-manifolds we start out with. By using this we establish the existence of an absolutely exotic contractible Stein manifold pair, and absolutely exotic homotopy $S^1\times B^3$  Stein manifold pair.
\end{abstract}

\date{}
\maketitle

\setcounter{section}{-1}

\vspace{-.3in}

\section{Introduction} \label{introduction}

\vspace{.1in}

Here we will prove the following theorems which strengthens \cite{ar}:

\begin{Thm}\label{Theorem1}
There is a  pair of compact contractible Stein $4$-manifolds $W_1$, $W_2$, which are homeomorphic but not diffeomorphic to each other. 
\end{Thm}

As in the examples of \cite{ar}, $W_1$ and $W_2$ are related to each other by cork twisting along a cork $(W,f)\subset W_{i}$, $i=1,2$ (e.g. \cite{a}). It follows from the construction that $W_{1}$ and  $W_{2}$ can not be corks (see \cite{ar1}). 

\begin{Thm}\label{Theorem2}\label{Theorem2}
There is a  pair of compact Stein $4$-manifolds $Q_1$, $Q_2$ homotopy equivalent to $S^1\times B^3$, which are homeomorphic but not diffeomorphic to each other. 
\end{Thm}

Similarly, $Q_1$ and $Q_2$ are related to each other by``anticork'' twisting (e.g. \cite{a}) along an anticork contained in $(W,f)$. The important feature of our theorems is that the manifolds constructed are all Stein. This property enables us to imbed them into closed symplectic manifolds and relate their exoticity to the nonvanishing gauge theory invariants of closed symplectic manifolds (\cite{t}). The proofs of these theorems will be based on:

\begin{Thm}\label{Theorem3} 
Any knot $K$ in $S^1\times S^2$, which is  freely homotopic to $S^1\times pt$, is  invertibly concordant to $S^1\times pt$.
\end{Thm}

\section{Background} \label{background}

Let us recall some basic definitions about knot concordances:

\begin{Def}\label{Definition1}
Two knots $K_1$ and $K_2$ are said to be \emph{concordant} if there is a smooth proper embedding of an annulus $F: S^1\times [0,1] \hookrightarrow Y\times[0,1] $, such that its boundary is $\partial{F(S^1\times [0,1])} = K_1 \times 0 \sqcup (-K_2) \times  1 $,  where $-K_2$ is  the knot $K_2$ with the reversed orientation.
\end{Def}

\begin{Def}\label{Definition2}\cite{s} A concordance $F$ between knots $K_1$ and $K_2$ is said to be \emph{invertible} if there is a concordance $F'$ from $K_2$ to $K_1$ such that $F \cup F': S^1\times [0,1] \hookrightarrow Y\times[0,1] $ is the product concordance $K_1 \times [0,1] \subset Y\times[0,1]$. In this case, we say \emph{$K_1$ is invertibly concordant to $K_{2}$,  and  $K_2$ splits $K_1\times [0,1] $}. In particular when $Y = S^3$ and $K_1$ is the unknot then $K_2$ is called \emph{doubly slice}.
\end{Def}

\begin{Def}\label{Definition3} 
An invertible cobordism $X$ from $M$ to $N$ is a smooth
manifold with $\partial X = M \sqcup -N$, such that there is a cobordism $X'$  with
$\partial X'  = N \sqcup -M $ and $X \underset{N}{\sqcup} X' = M \times I $.
\end{Def}
 
Following Proposition follows from computations by Snap{P}y \cite{cdgw}, and  verifications of these computations are due to \cite{dhl} and \cite{hikmot}.

\begin{Prop}\label{Proposition4}
The knot $K\subset S^{1}\times S^{2}$ given in the Figure~\ref{c1} is hyperbolic with trivial symmetry group.
\end{Prop}

\begin{figure}[ht]
 \begin{center}  
\includegraphics[width=.3\textwidth]{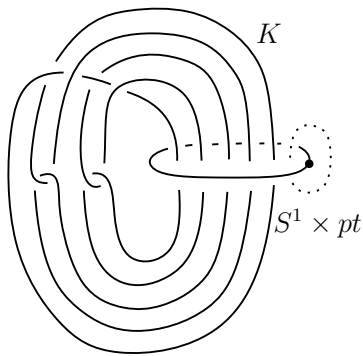}   \caption{The knot $K\subset S^{1}\times S^{2} $ has hyperbolic complement and trivial symmetry group} 
\label{c1} 
\end{center}
\end{figure}

 By using Theorem ~\ref{Theorem3} and Proposition~\ref{Proposition4} we will convert relatively exotic smooth structures on compact $4$-manifolds to absolutely exotic smooth structures, by using the construction given in \cite{ar}. Our goal is to generate Stein exotic examples that can not be obtained from \cite{ar}.

Now recall the construction of absolutely exotic smooth structures on compact 4-manifolds, from relatively exotic structures:

\begin{Thm}[\cite{ar}]\label{AR}
Let  $F : W \rightarrow W$ be self homeomorphism of a compact smooth $4$-manifold, whose restriction to $ \partial W$ is a diffeomorphism which does not extend to a self diffeomorphism of $W$. Then $W$ contains a pair of homeomorphic smooth $4$-manifolds $V$ and $V'$\ homotopy equivalent to $W$, but $V$ and $V'$ are not diffeomorphic to each other.
\end{Thm}

The construction of \cite{ar} relies on finding a knot in $S^3$ which is doubly slice, hyperbolic, and with no symmetry. Here we will remind the original construction, and show that we can use invertible knot concordances in $S^1 \times S^2$ between a knot $K$ and  $S^1\times pt$, where $K$ is a hyperbolic knot with trivial symmetry group, splitting the concordance. 

\vspace{.03in}

\begin{Lem}

Let $K$ be a  knot in $S^1\times S^2$, which is  freely homotopic to $S^1\times pt$ with exterior  $X$. Then $H_{*}(X) = H_{*}(S^1\times D^2) $. 

\end{Lem}

\begin{proof} Follows from Mayer-Vietoris sequence.
\end{proof}

\begin{Cor}\label{concordance complement}
Let $K_1$, $K_2$ be knots in the free homotopy class of the $S^1\times pt$ in $S^1\times S^2$, and  $C$ be a concordance between $K_1$ and $K_2$, then $H_*(S^1\times S^2\times I - \nu(C)) = H_*(S^1 \times D^2 \times I) $ and generated by the $S^1\times pt$.
\end{Cor}

Note that a tubular neighbourhood $\nu(C)$ of $C$ in $S^1\times S^2 \times I$ is an embedding of $S^1\times D^2 \times I$ in $S^1\times S^2 \times I$, where the image of $S^1\times 0 \times I$ is $C$. The image of ${pt}\times \partial D^2 \times t$ is called a meridian of $C$, then a meridian of $C$ bounds a disk in the tubular neighbourhood of $C$, and it is homologous to zero in the exterior. Similarly we call the image of $S^1\times pt \times t$ as a longitude. %Note that any two longitude-meridian pairs are homotopic in the neighbourhood of $C$ and in the complement of $C$. %

\vspace{.05in}

The following lemma is adapted from \cite{ar}. We restate it here in a slightly different way to show that we can use appropriate concordances of knots in $S^1\times S^2$ instead of in $S^3$. The only difference here is that the gluing map $\phi $ identifies meridian to meridian not to longitude. 
%We write here an explicit proof. 

\newpage

\begin{Lem}
Suppose that $\gamma$ is a framed knot in a closed $3$-manifold $M$, and $C$ is an invertible concordance from the $S^1\times pt$ to the knot $K$ in $S^1\times S^2 \times I$. Define
$$ H = \left( M \times I - \left( \gamma \times D^2 \times I\right) \right) \underset{\phi}{\bigcup} \left( S^1 \times S^2 \times I -  \left( C \times D^2\right) \right)  $$
where $\phi: T^2 \times I \to T^2 \times I$ is a diffeomorphism which sends the meridian of $K$ to the meridian of the knot $\gamma$. Then $H$ is an invertible homology cobordism from $M$ to  $N$. If $\pi_1(S^1\times S^2\times I - C\times D^2) \cong \Z$, then the inclusion $\iota : M \hookrightarrow H$ induces an isomorphism on fundamental groups.
\end{Lem}

\section{Proof of Theorem 3 } \label{invertible}

\begin{proof}[Proof of Theorem~\ref{Theorem3}]

 Start with a knot $K \subset S^1\times S^2$ in the concordance class of $S^1\times pt$. Built a concordance $F$ between $K$ and $S^1\times pt$ as it is explained in \cite[Theorem 2]{y}. $F$ is properly embedded in $S^1 \times S^2\times I$, the last coordinate is the time, and we have $S^1\times S^2$ at every level. From the construction we see the knot $K$ at the top level, and we perform genus zero cobordism, i.e. we attach bunch of bands $\{b_i\}_{i=1}^{m}$, turning $K$ to $S^1\times pt$ union disjoint unknots $\{u_i\}_{i=1}^{m}$ linking $S^{1}\times pt$. For $1 > t_i > t_{i-1}>0 $, following levels are depicted partially in Figure~\ref{c2}

 \begin{itemize}
 
 \item $F \cap (S^1\times S^2 \times 1) = K $
 
 \item $F \cap (S^1\times S^2 \times t_i) = K \cup b $ 
 
 \item $F \cap (S^1\times S^2 \times t_{i-1}) = u \sqcup (S^1\times pt)$
 
 \item $F \cap (S^1\times S^2 \times 0 ) = S^1 \times pt $

 \end{itemize}
 \begin{figure}[ht] \begin{center}  
\includegraphics[width=.57\textwidth]{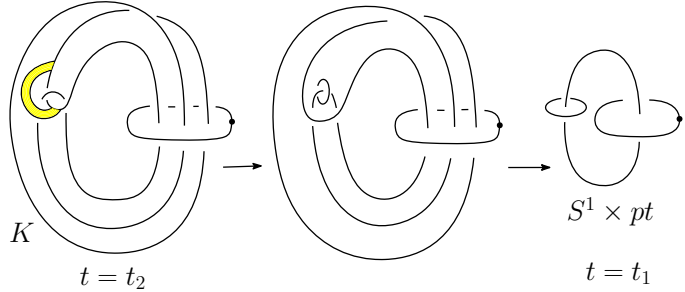}
\caption{The construction of the concordance} 
\label{c2} 
\end{center}
\end{figure} 
 \noindent where each $u_i =\partial D^2$ is an unknot in $S^{1}\times S^{2}$ which is the boundary of the corresponding $0$-handle of $F$, and $b_{i}$ represents a band which is a $1$-handle attached to the surface. $F'$ in the Definition~\ref{Definition2} is $-F$ here.

\vspace{.05in}

Next will be to construct a handlebody decomposition for the concordance complement  $S^1\times S^2\times I - \nu(F\cup (-F))$ from the handlebody decomposition of the surface $F$. For a detailed discussion of handlebody decomposition of surface complement in $4$-manifolds one can consult Chapter 1.4 of \cite{a}. To construct the complement of $F$ in $S^1\times S^2 \times I $ start from the bottom. First we see complement of $S^1 \times pt \times [0,t]$ in $S^1 \times S^2 \times [0,t] $ which is $S^1 \times D^2 \times [0,t]$. Then as $t$ increase,  $0$-handles of the surface $F$ appears, in the complement which corresponds carving properly embedding of disks from $B^4$. So we have connected sum of $m+1$ copies  of $S^1\times B^3$. At the last step the complement gains $m$ many $4$-dimensional $2$-handles for the $1$-handles $b_i$ of the surface, as in the middle picture of  Figure~\ref{c3}.  Alternatively the reader should compare this to \cite{y}, where the cobordism was constructed from top to bottom. For example, Figure~\ref{c3} describes the concordance from the Mazur knot to $S^{1}\times pt$

\vspace{.05in}

To double the concordance complement, we attach upside-down handles along the dual of the $2$-handles (with zero framing), and dual $3$-handles as upside-down $1$-handles. We use the dual $2$-handles to get rid of the self-linking of the original $2$-handle, therefore ending up with cancelling $1/2$ and $2/3$ handle pairs. After cancellations the complement becomes a product. In Figure~\ref{c3}, reader can verify that sliding over the dual $2$-handle ($0$-framed little circle linking the $2$-handle) will make the $2$-handle trivially link the $1$-handle. 

\begin{figure}[ht] \begin{center}  
\includegraphics[width=.55\textwidth]{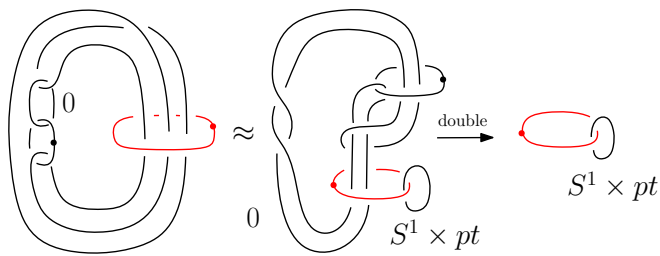}
\caption{The Construction of the concordance complement}
\label{c3}
\end{center}
\end{figure} 

All these cancellations happens away from the endpoints \\ $S^1\times S^2\times \{0\} - \nu(S^1\times \{pt\}) $ and  $S^1\times S^2\times \{1\} - \nu(S^1\times \{pt\})$ provided  attaching circles of the two handles are away from the $S^1\times pt$. Note that anything links to $S^1\times pt$ can be unlinked by sliding over the  middle (red)  $1$-handle.
So far we have constructed a diffeomorphism from $S^1 \times S^2\times I - \nu(F\cup (-F))$ to  $S^1 \times S^2\times I - \nu(S^1\times pt\times I) = S^1\times B^2 \times I $, which is identity near the endpoints. Next, we prove that 
this diffeomorphism extends to self diffeomorphism of $S^1 \times S^2\times I$ which takes the surface $F\cup -F $ to the product  $S^1 \times \{pt\} \times I$.

The diffeomorphism above induces a boundary diffeomorphism on $\partial (S^1\times B^2 \times I)= S^1 \times S^1 \times I \cup (S^1\times B^2 \times 0  \;\sqcup \;S^1\times B^2 \times 1) $. When we restrict this to the partial boundary we get a diffeomorphism $\tau: T^2\times I \rightarrow T^2\times I$ which is identity near $T^2\times \{0,1\}$. By Lemma 3.5 of Waldhausen \cite{w}, $\tau$ is isotopic to the identity map rel boundary hence it extends to tubular neighbourhoods of the surfaces. Hence we have a self diffeomorphism of $S^1 \times S^2\times I$ taking $\nu(S^1\times pt \times I)$ to  
$\nu(F\cup (-F))$. So diffeomorphic complements uniquely determine the surfaces. \end{proof}

Theorem $3$ is related to Light Bulb theorems, that is after doubling the concordance, we have a surface in $S^1\times S^2 \times I$ which is bounded by $S^1\times pt \times i$ $(i=0,1)$. By capping both boundaries with $D^2 \times S^2$ s, we get a surface in $S^2\times S^2$ intersecting $pt\times S^{2}$ at one point, then apply \cite{l} and \cite{g}. The advantage of this construction is, it gives  a rel boundary diffeomorphism  $(S^{1}\times S^{2}\times I, F\cup -F) \to (S^{1}\times S^{2} \times I, S^{1}\times pt\times I)$.

\vspace{.03in}

By applying the technique of  Section~\ref{invertible}  to the curve $K\subset S^{1}\times S^{2}$ of Figure~\ref{c1}, we get the  Figure~\ref{c4},  where the curve $K$ now looks like the standard linking circle to the $1$-handle (dotted curve in the figure).

\begin{figure}[ht]  
\begin{center}  
\includegraphics[width=.23\textwidth]{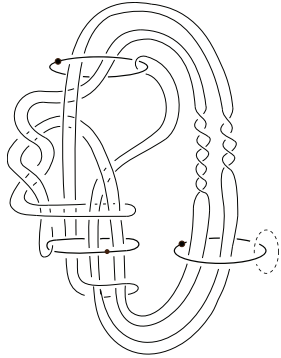}
\caption{The concordance complement in $S^{1}\times S^{2}\times I$}
\label{c4}
\end{center}
\end{figure}
In Figure~\ref{c4} the concordance from $S^{1}\times pt$  to K is fully visible (i.e. the handles of the complement of $K$ are visible). Clearly cancelling the handles of Figure~\ref{c4} gives Figure~\ref{c1} (the dotted circle corresponds to $K$). 

%\vspace{-.15in} 

\subsection{Proof of Theorem 1} First we proceed as in \cite{ar}, i.e. glue the homology product cobordism $H$ obtained from Figure~\ref{c4} to the boundary of the cork $W$ (as discussed in Section~\ref{background}). From this we get a new contractible manifold $W_1$ containing $W$, such that performing the cork twisting $W\subset W_1$ gives us an absolutely exotic copy $W_2$ of $W_1$.  Recall from Section~\ref{background} we construct  $W_1=W\cup H$ by gluing W and the homology product cobordism $H$  along the longitudes of  $\eta \subset W$ and $K\subset S^{1}\times S^{2} $, which is the Figure~\ref{c5}. Here we are using the  {\it roping} technique of \cite{a} to draw the handlebody of $W_{1}$. From the discussion above, Figure~\ref{c5} is equivalent to Figure~\ref{c6}. By zero and dot exchanges to $W$ inside $W_{1}$ gives $W_{2}$.The only remaining issue is to put Stein structure on $W_1$ and $W_2$ (they are homeomorphic by Freedman's theorem).

\begin{figure}[ht] 
\begin{center}  
\includegraphics[width=.6 \textwidth]{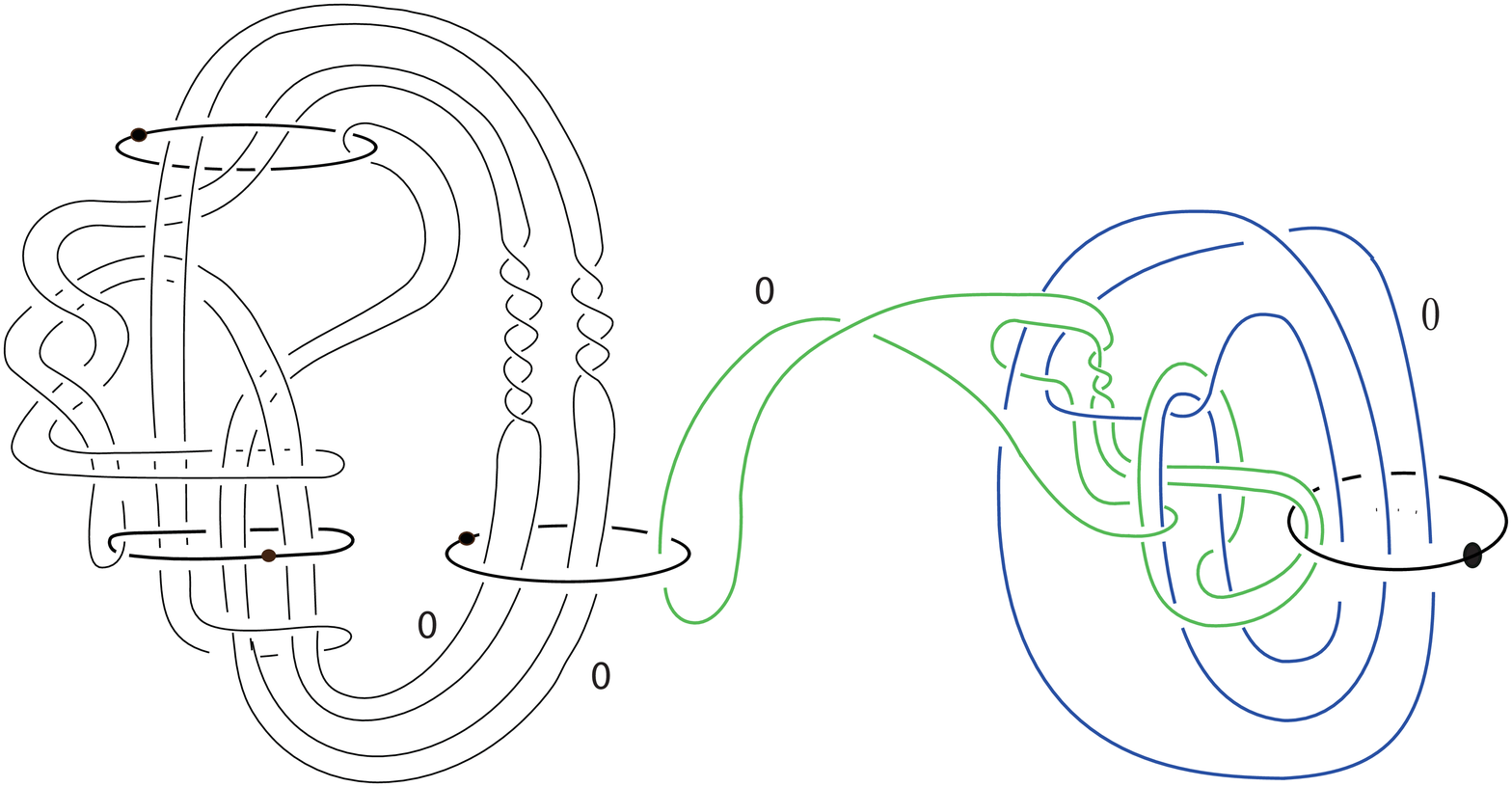}
\caption{$W_{1}$} 
\label{c5}
\end{center}
\end{figure}

\begin{figure}[ht]
\begin{center}  
\includegraphics[width=.5\textwidth]{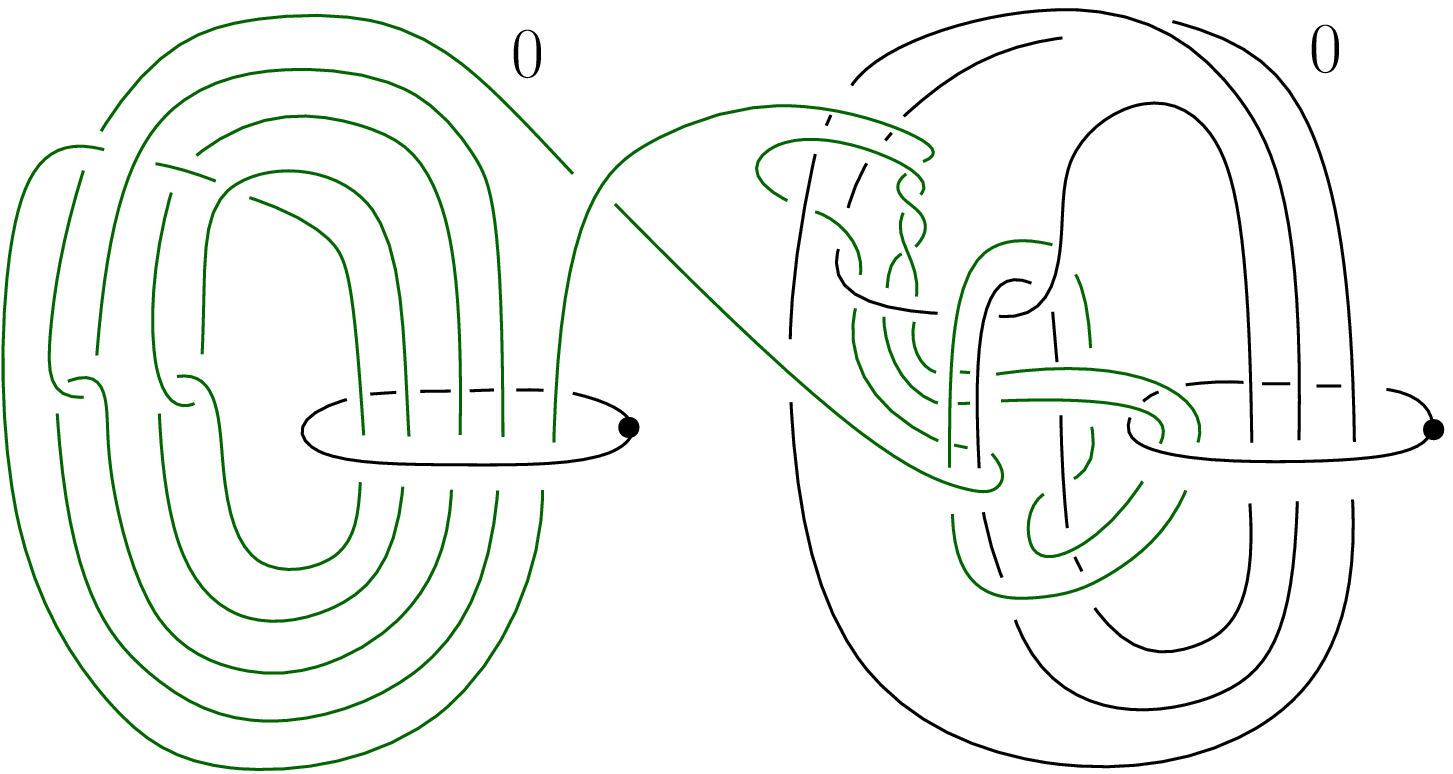} 
\caption{$W_{1}$} 
\label{c6}
\end{center}
\end{figure}

\vspace{05in}

For this we will work with the  picture of  $W_{1}$, given in Figure~\ref{c6}. We first put  Figure~\ref{c1} in Legendrian position and get Figure~\ref{c7}, where the $2$-handle has $tb=4$. 

\vspace{0.3in}

\begin{figure}[ht] 
\begin{center}  
\includegraphics[width=.45\textwidth]{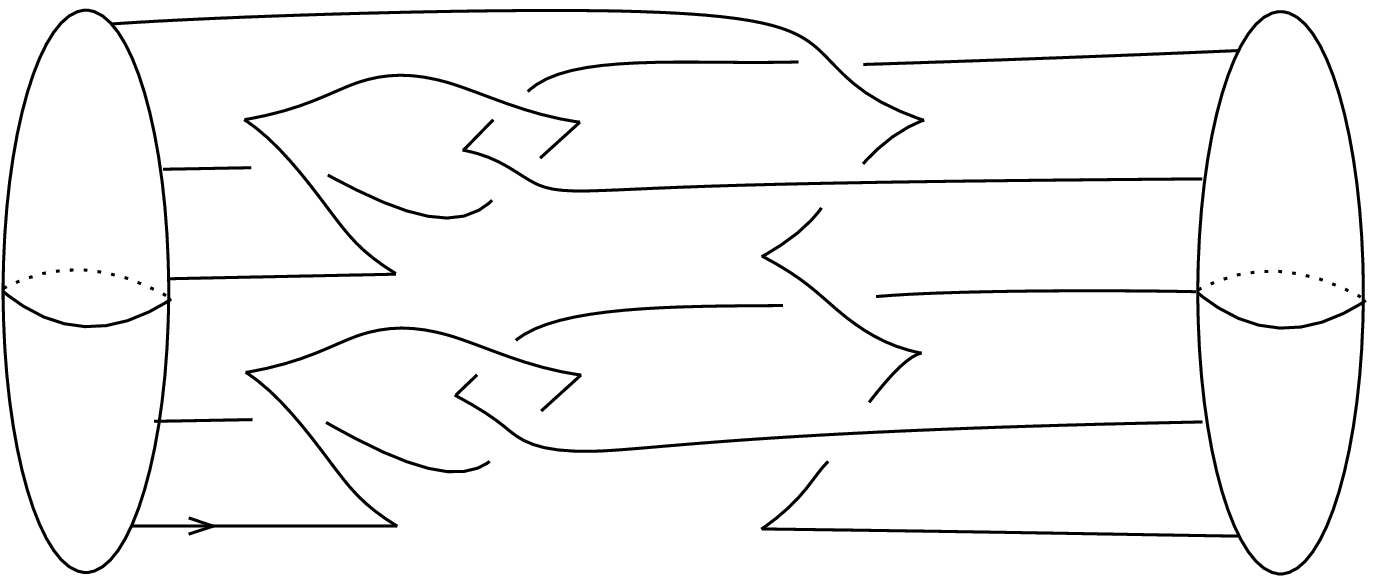}
\caption{tb=4} 
\label{c7}    
\end{center}
\end{figure}

\begin{figure}[ht] 
\begin{center}  
\includegraphics[width=.7\textwidth]{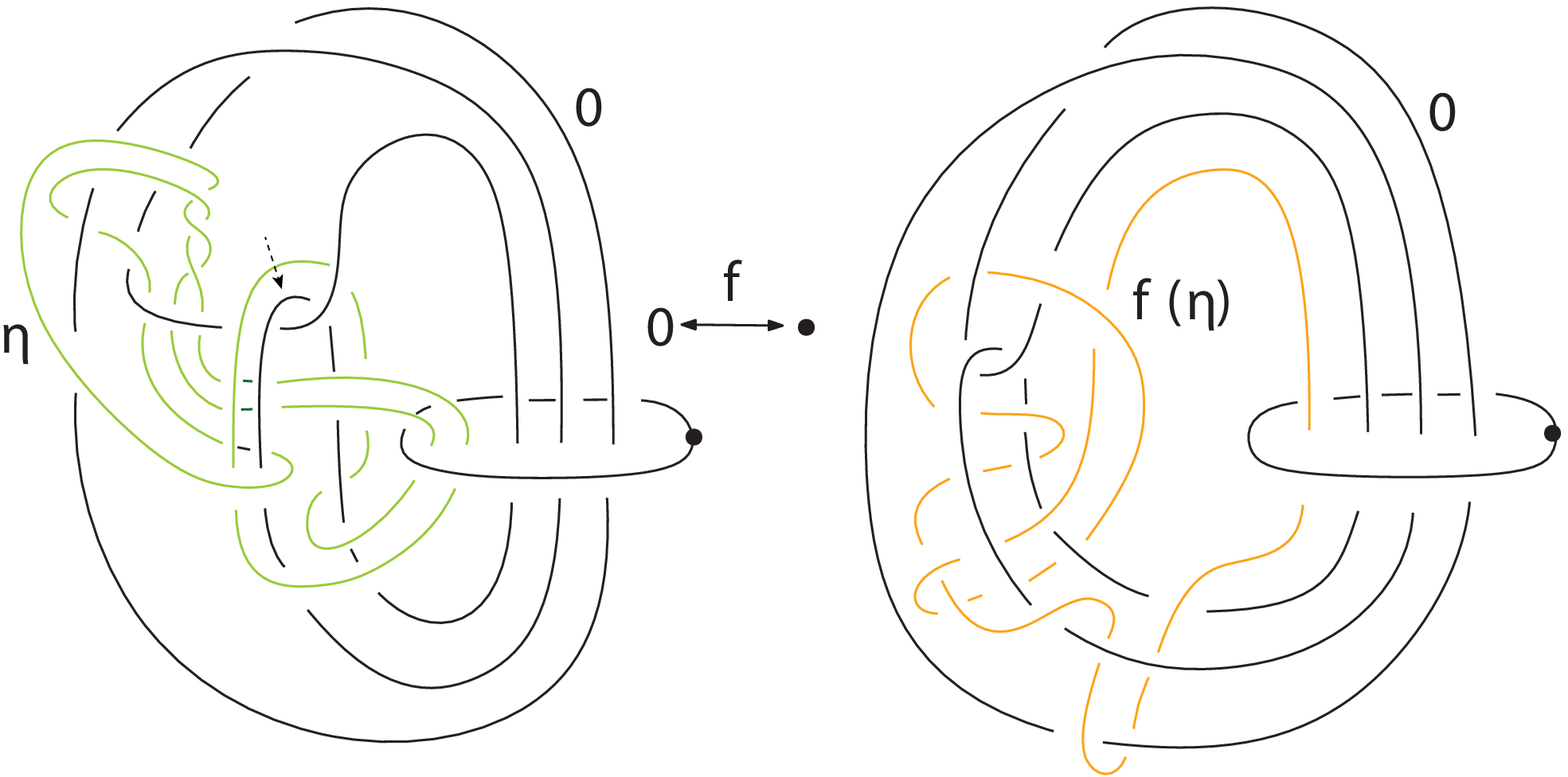}  
\caption{}
\label{c8} 
\end{center}
\end{figure}

The two pictures of  Figure~\ref{c8}  are related to each other by zero and dot exchanges.  Then we put both handlebodies of Figure~\ref{c8} in Legendrian position, and connected sum with Figure~\ref{c7}. This gives Figures~\ref{c11} and \ref{c12}, which are pictures of $W_{1}$ and $W_{2}$ where both are Stein. \qed

\begin{figure}[ht]
\centering
\begin{minipage}{.45\textwidth}
\centering
\includegraphics[width=.85\linewidth]{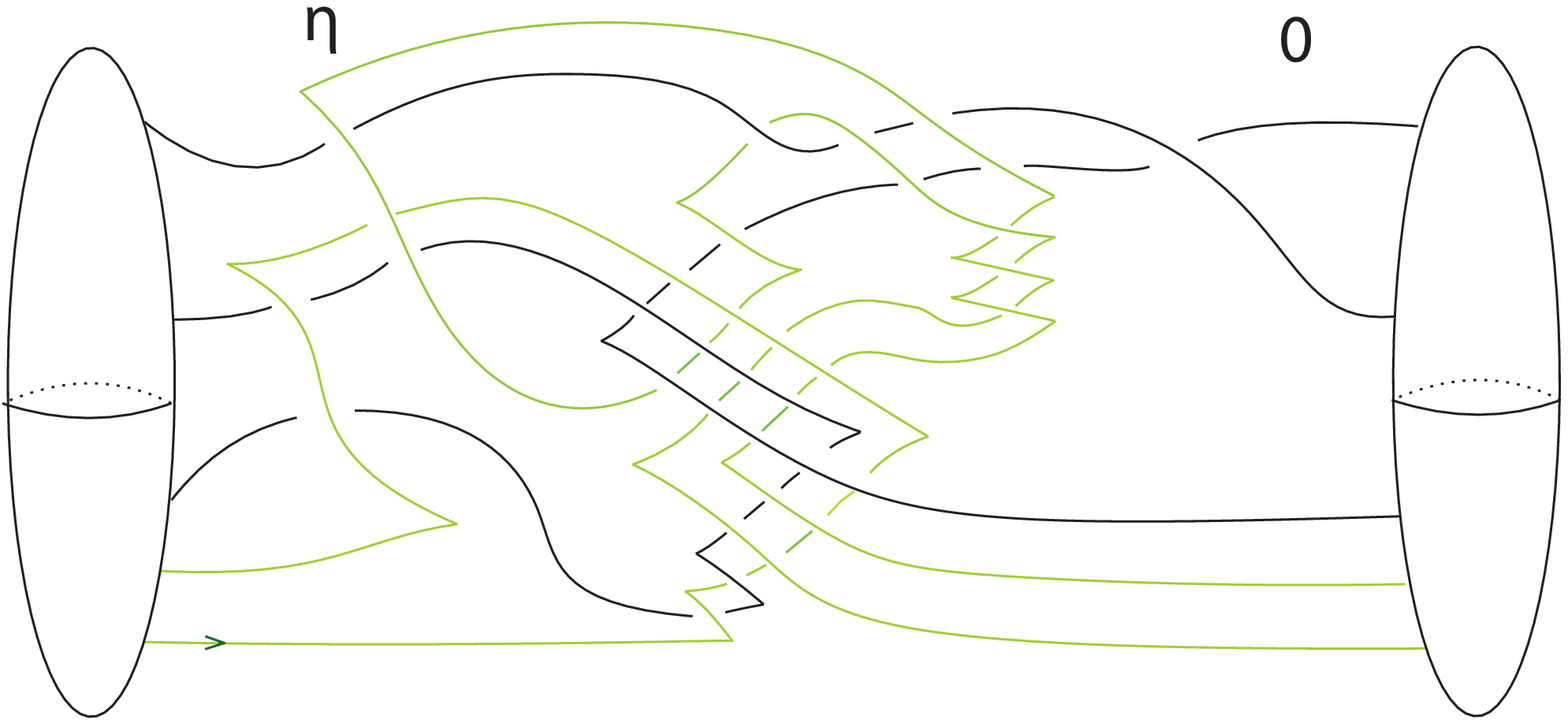}
\caption{$tb(\eta)=-3$}
\label{c9}
\end{minipage}%
\begin{minipage}{.45\textwidth}
\centering
\includegraphics[width=.85\linewidth]{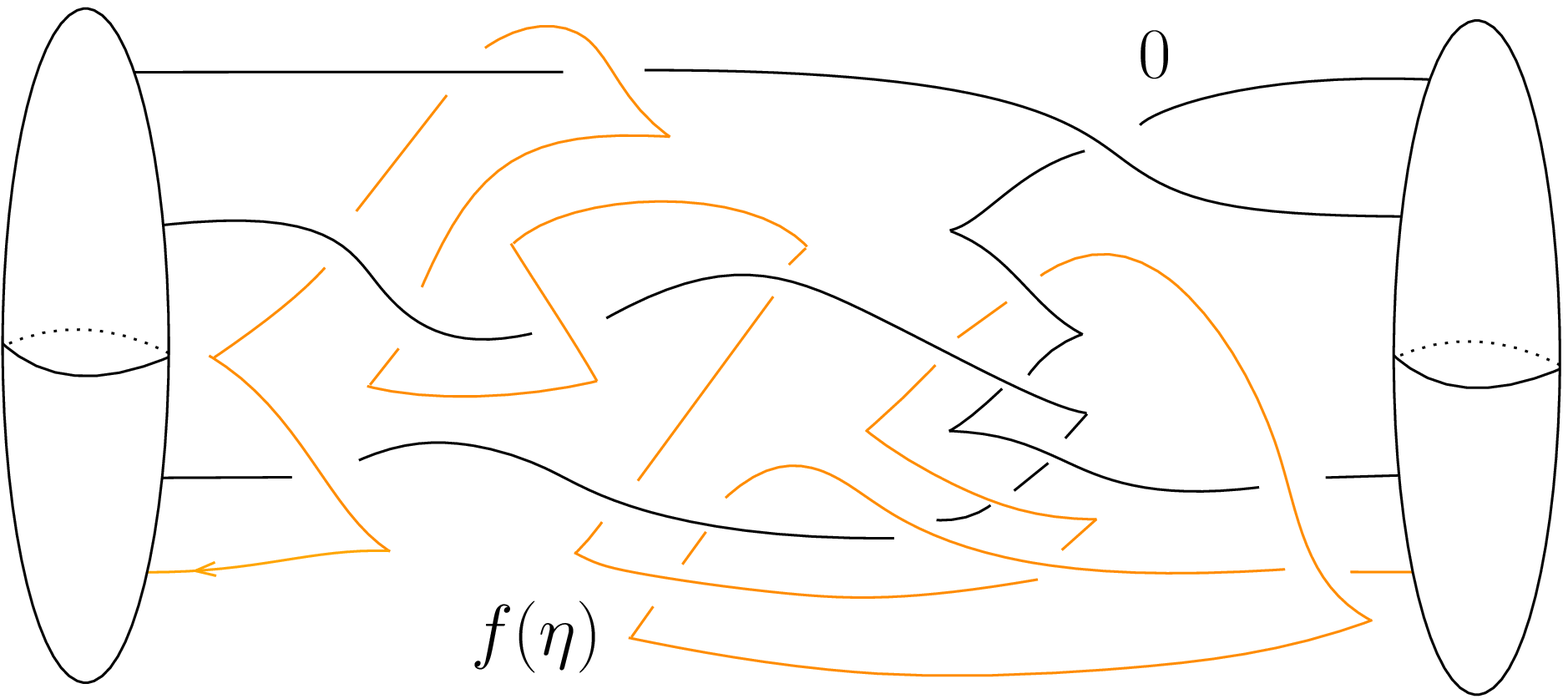}
\caption{$tb(f(\eta))=-3$}
\label{c10}
\end{minipage}
\end{figure}

\begin{figure}[ht]
\begin{center}  
\includegraphics[width=.42\textwidth]{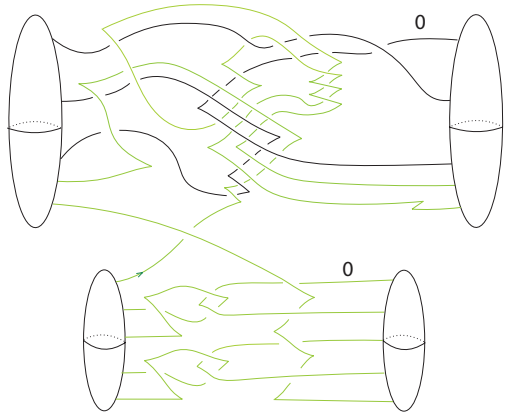} 
\caption{$W_{1}$, as a Stein handlebody} 
\label{c11}
\end{center}
\end{figure}

\begin{figure}[ht]
\begin{center}  
\includegraphics[width=.42\textwidth]{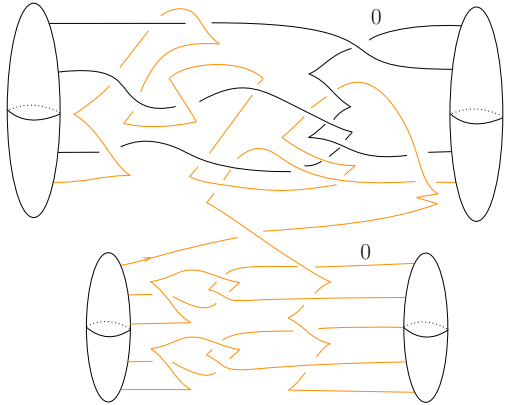}
\caption{$W_{2}$, an exotic copy of $W_{1}$, as a Stein handlebody}
\label{c12}  
\end{center}
\end{figure}

\clearpage

\subsection{Proof of Theorem 2} As above, we will use the construction of \cite{ar} such a way that the end result will be a pair of Stein manifolds, homotopy equivalent to $S^1$, which are absolutely exotic copies of each other. Now recall the construction of \cite{a1} (and \cite{a}): The knot $R\subset S^3$ on the left of Figure~\ref{c13} bounds a ribbon $D^{2}\subset B^{4}$  in two different ways.  Both of these ribbon complements $C^4=B^4 - N(D)$ are diffeomorphic to each other, where $N(D)$ denotes the tubular neighborhood of $D$. Furthermore, $C^4$ is homotopy equivalent to $S^1$, and the diffeomorphism  $f: \partial C \to \partial C$ is described by the involution  of Figure~\ref{c13}   (the ``zero-dot exchange" of the two middle pictures of Figure~\ref{c13}) does not extend to a diffeomorphism $C\to C$, but does extend to a hemeomorphism. Note that $f$ is the anticork involution in 10.2 of \cite{a}. 

\begin{figure}[ht] \begin{center}  
\includegraphics[width=.87\textwidth]{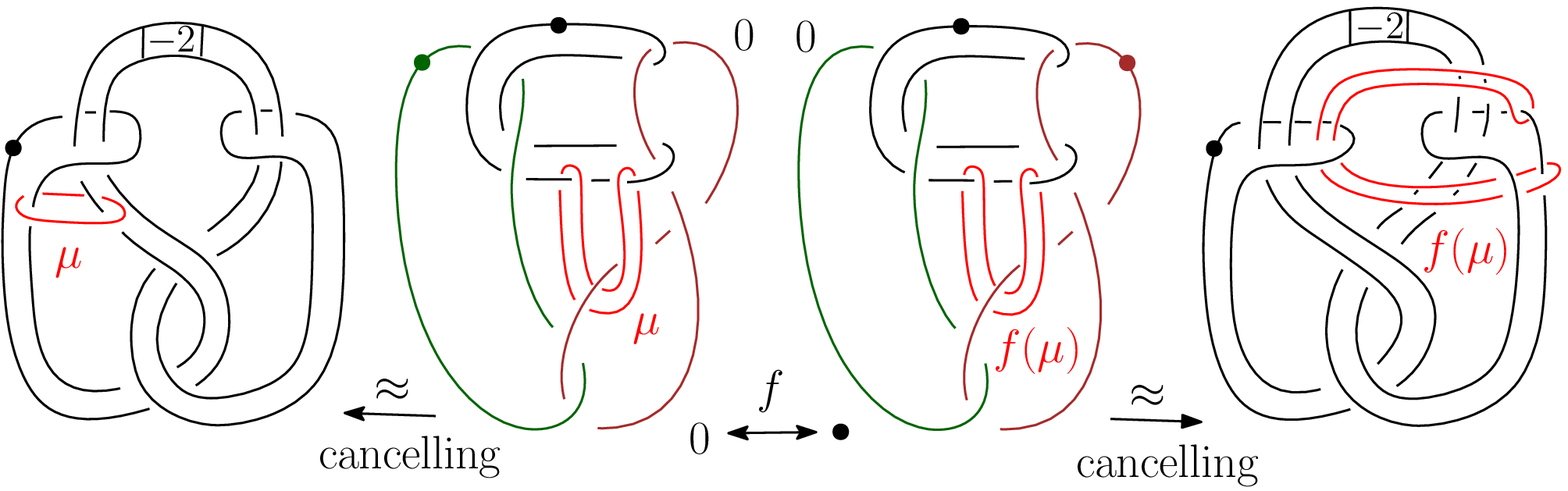}
\caption{$f: \partial C \to \partial C$} 
\label{c13}   
\end{center}
\end{figure}

Next we observe that $C^4$ is a Stein manifold by drawing the first picture of Figure~\ref{c13} as in Figure~\ref{c14} (this was first verified by Luke Williams). Then  pick a knot $\mu\subset \partial C^4$ with hyperbolic complement and trivial symmetry (again by verifying by SnapPy \cite{cdgw}). The first and last pictures of Figure~\ref{c13} becomes  Figures~\ref{c14} and  Figure~\ref{c15}.

\begin{figure}[ht]
\centering
\begin{minipage}{.45\textwidth}
\centering
\includegraphics[width=.75\linewidth]{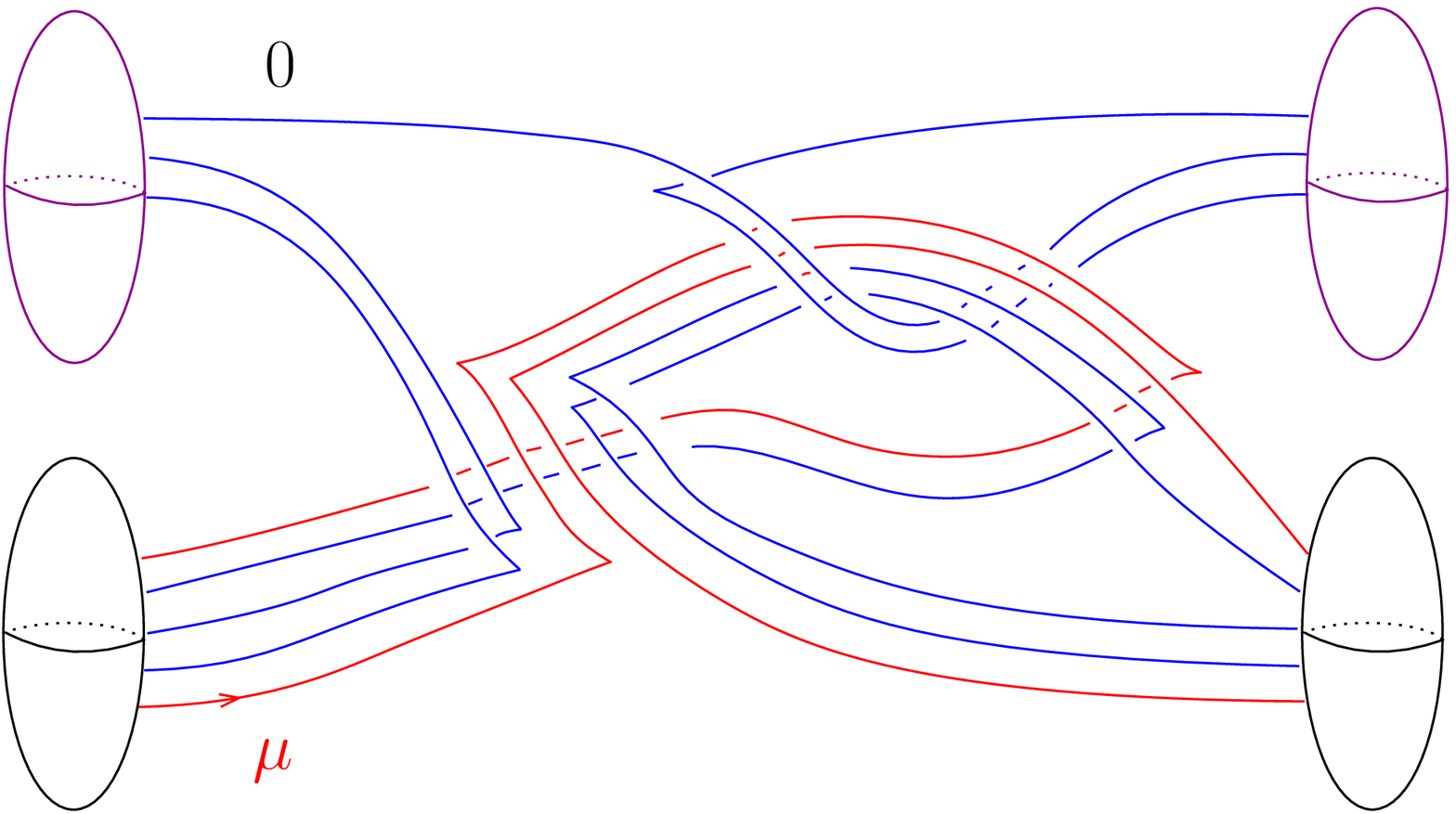}
\caption{$tb(\mu)=-1$}
\label{c14}
\end{minipage}%
\begin{minipage}{.45\textwidth}
\centering
\includegraphics[width=.75\linewidth]{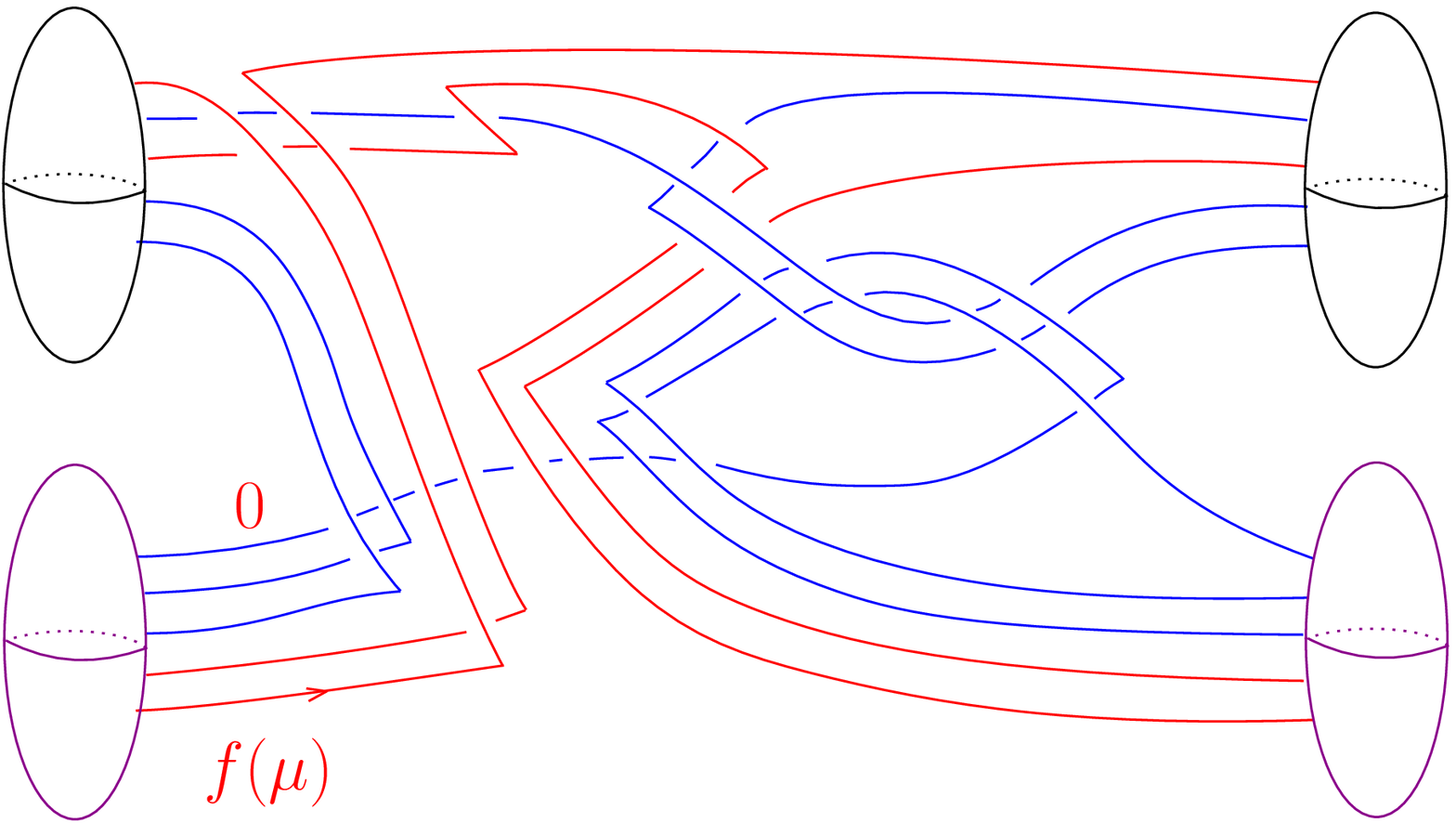}
\caption{$tb(f(\mu))=-3$}
\label{c15}
\end{minipage}
\end{figure}

 Finally by the roping technique of \cite{a}, we glue $H$ to $C$ by first identifying $K$ and $\mu$, then identifying $K$ and  $f(\mu)$ getting the manifolds $Q_1$ and $Q_{2}$ in Figures~\ref{c16} and \ref{c17}. Clearly both are Stein manifolds, and are homeomorphic to each other by Freedman's theorem (calculations $\pi_{1}(Q_{i})\cong\Z$ for $i=1,2$ can be checked from the pictures). \qed

\begin{figure}[ht] \begin{center}  
\includegraphics[width=.43\textwidth]{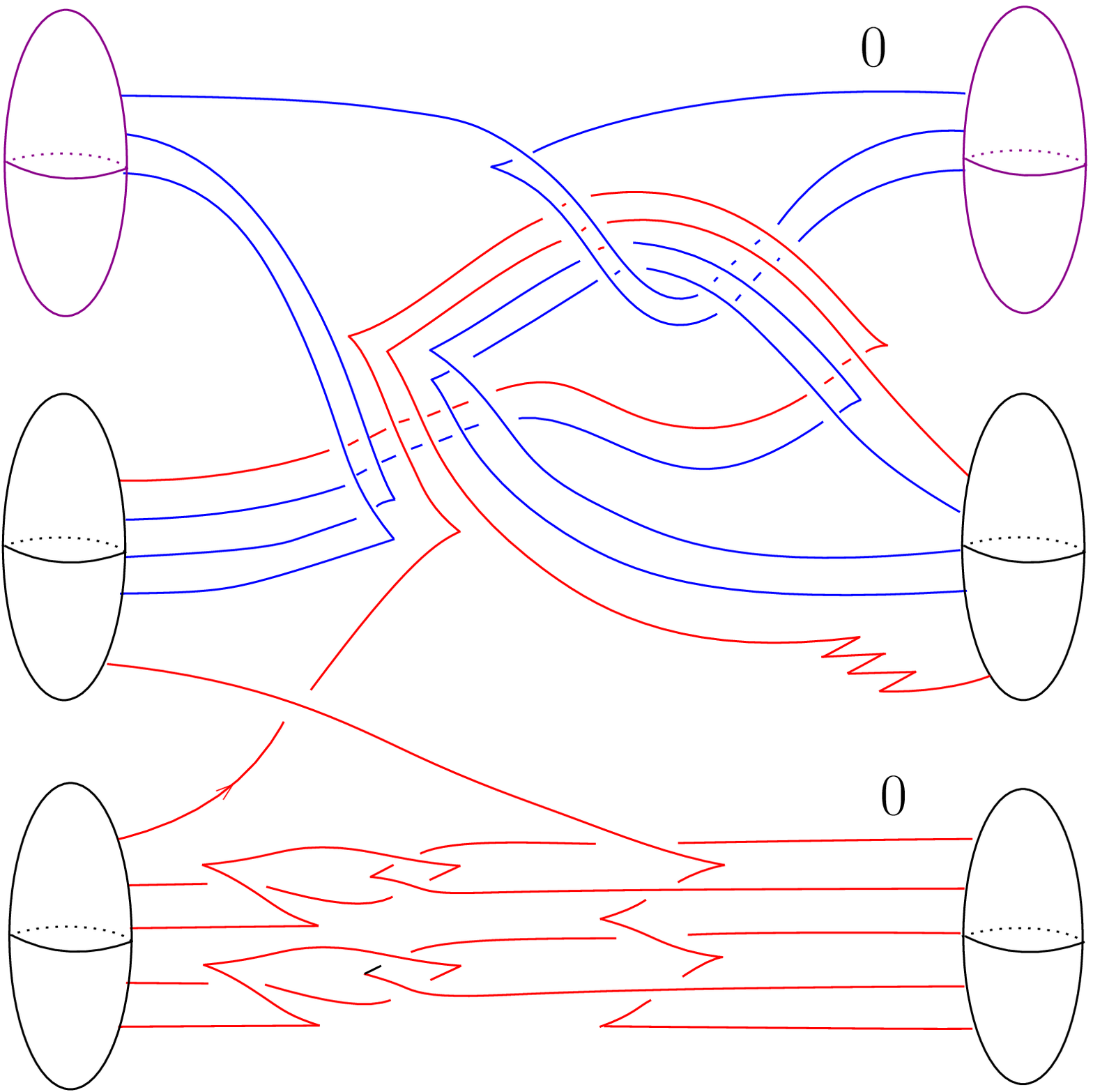}
\caption{$Q_{1}$} 
\label{c16}    
\end{center}
\end{figure}

\begin{figure}[ht] \begin{center}  
\includegraphics[width=.43\textwidth]{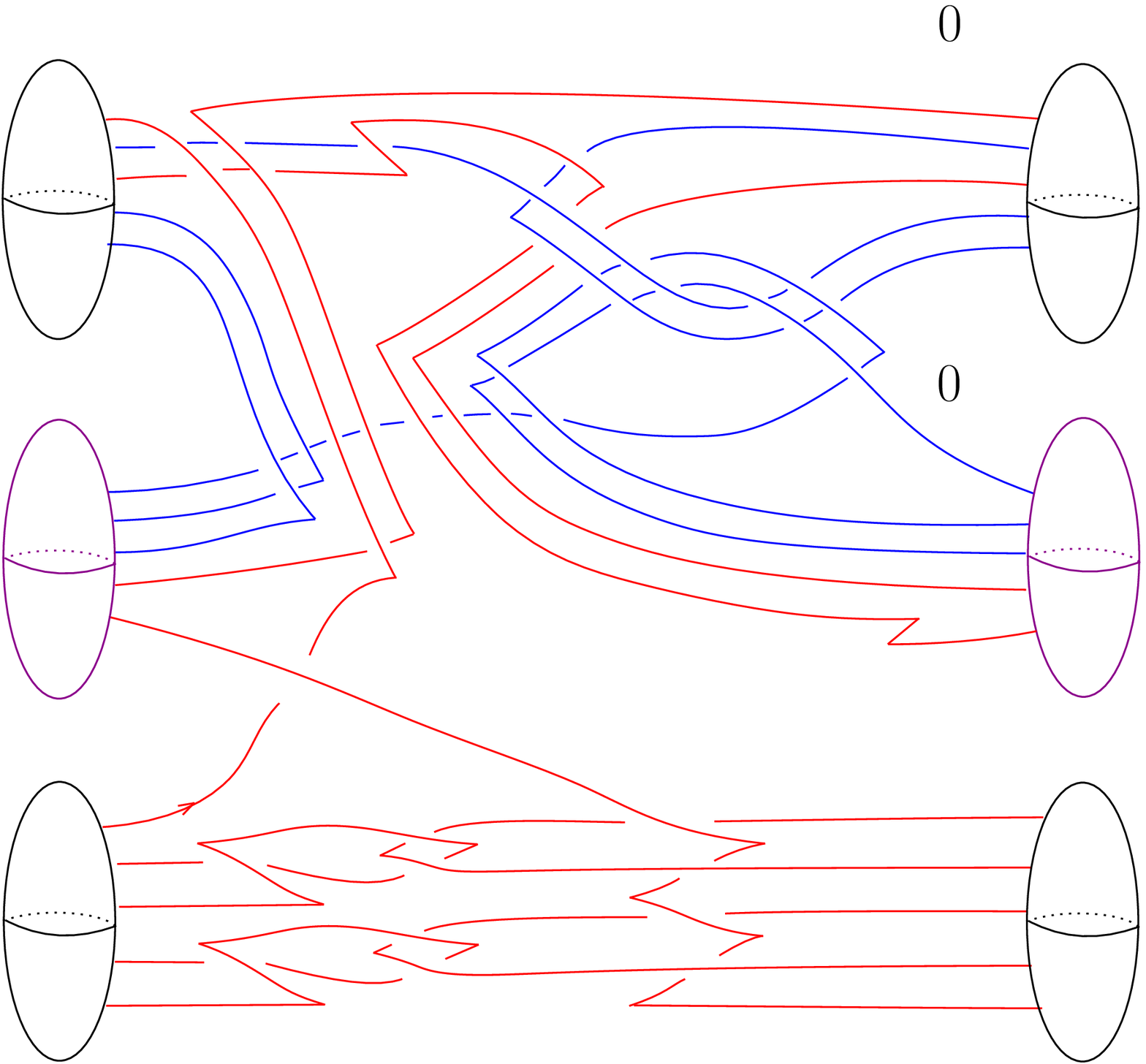}
\caption{$Q_{2}$} 
\label{c17}    
\end{center}
\end{figure}

\newpage

\begin{Rm} Notice each of the contractible manifolds $W_1$ and $W_2$ can be represented by a single 1/2-handle pairs. To see this in Figure~\ref{c12} cancel the 1-handle, at top of the picture, with the $2$-handle which goes through it geometrically once. To see this in Figure~\ref{c11}, we first slide the colored $2$-handle over the other $2$-handle (which corresponds the handle slide indicated by the arrow in the first picture of Figure~\ref{c8}) then the resulting $2$-handle in Figure~\ref{c11} goes through the $1$-handle once. This means that if $\gamma_i$  are the  linking circles of the $2$-handles of $W_i$, for $i=1,2$, there is no diffeomorphism $f:\partial W_1\to \partial W_2$ taking $\gamma_1$ to $\gamma_2$; if there was, we can extend it to a diffeomorphism $W_1\to W_2$ by using ``carving'' (Section 2.5 of \cite{a}). It was pointed out to us that this provides an answer to Problem 1.16 in \cite{k}, because surgering each 
$\partial W_{i}$ along $\gamma_{i}$ results $S^1\times S^2$ (since posting of our paper this fact has also been pointed out in \cite{kp} and \cite{hmp} as well).
\end{Rm}

\noindent{\it Acknowledgement.} First named author would like to thank Danny Ruberman for enjoyable collaboration in \cite{ar}  which was a motivation of this paper. The second named author would like to thank Nathan Dunfield, Nick Ivanov and Kouchi Yasui for valuable conversations.

\end{document}